\documentclass[a4paper]{article}
\usepackage{hyperref,authblk,amsmath,amssymb,amsthm,color}
\usepackage{pgf,tikz}
\usepackage{mathrsfs}
\usetikzlibrary{arrows}
\usepackage{circuitikz}

\title{Resistors in dual networks}
\author{Martina Furrer} 
\author{Norbert Hungerb\"uhler} 
\author{Simon Jantschgi}
\affil{Department of Mathematics, ETH Z\"urich, 8092 Z\"urich, Switzerland}
\date{}
\newtheorem{theorem}{Theorem}
\newtheorem{proposition}[theorem]{Proposition}

\theoremstyle{definition}
\newtheorem{definition}[theorem]{Definition}
\newtheorem{example}[theorem]{Example}

\begin{document}
\maketitle
\begin{abstract}
\noindent Let $G$ be a finite plane multigraph and $G'$ its dual. Each edge $e$ of $G$ is
interpreted as a resistor of resistance $R_e$, and the dual edge $e'$ 
is assigned the dual resistance $R_{e'}:=1/R_e$. Then the equivalent resistance $r_e$
over $e$ and the  equivalent resistance $r_{e'}$ over $e'$ satisfy
$r_e/R_e+r_{e'}/R_{e'}=1$. We provide a graph theoretic proof of this relation
by expressing the resistances in terms of sums of weights of spanning trees
in $G$ and $G'$ respectively.
\end{abstract}


\section{Introduction}
The systematic study of electrical resistor networks goes back to 
the german physicist Gustav Robert Kirchhoff in the middle of the 
19th century. In particular, Kirchhoff's two circuit laws and Ohm's law allow to fully
describe the electric current and  potential in a given static network
of resistors and voltage sources. In the course of his investigations, Kirchhoff discovered
the Matrix Tree Theorem, which states that the number of spanning 
trees in a graph $G$ is equal to any cofactor of the Laplacian matrix of $G$.
Surprisingly, this purely graph theoretic fact has a deep connection
to the physical question of the equivalent resistance between two
vertices of an electric network.


In the simplest situation, a finite simple graph $G$ can be interpreted as an electrical network
by considering each edge as a resistor of 1 Ohm. Then, one of Kirchhoff's results states that the 
equivalent resistance over an edge $e$ in this graph
is given by the quotient of the number of spanning trees containing
the edge $e$ divided by the total number of spanning trees in $G$.

More generally, we may consider a finite multigraph $G$ and assign to each edge $e$ of $G$ a weight $R_e>0$
interpreted as resistance of $e$. Also in this case, the equivalent resistance 
between two vertices can be  expressed  in terms of sums of weights of spanning trees (see
Section~\ref{sec-preliminaries}). 

Consider a cube as a graph with unit resistance on each edge
and the dual polyhedron, the octahedron, in the same way.
The equivalent resistance over an edge of the cube turns out to be
$7/12$ Ohm, the equivalent resistance over an edge of the octahedron
is $5/12$ Ohm. Observe, that these values add up to $1$!
The same phenomenon occurs for the dodecahedron with equivalent resistance of $19/30$ Ohm over each edge
and the dual graph, the icosahedron, with $11/30$ Ohm, or
the rhombic dodecahedron with equivalent resistance of $13/24$ Ohm over each edge
and the dual graph, the cuboctahedron, with $11/24$ Ohm.
This is not just a coincidence:
Suppose that a planar graph and its dual are both interpreted as electrical
networks with unit resistance for all edges. Now, if $r_e$ is the equivalent 
resistance over an edge $e$ and $r'_e$ is the equivalent 
resistance over the dual edge $e'$, then $r_e+r'_e=1$ (see~\cite[Exercise 7, Section 9.5]{bapat}, \cite[Theorem 2.3]{yang2013}).
The aim of this article is to generalize this formula to plane networks with 
arbitrary resistors (see Theorem~\ref{thm-main})
and to give a graph theoretic proof.
\section{Preliminaries}\label{sec-preliminaries}
Let $G$ be a finite connected graph with $n\ge 3$ vertices and
without loops. Multiple edges are allowed. Each edge $e$ is
considered as a resistor of resistance $R_e>0$. Then, we consider the
weighted Laplace matrix $L=(\ell_{ij})$ of $G$ defined as\footnotetext[1]{By convention, an empty sum is $0$.}
$$
\ell_{ij}:=\begin{cases}
-\sum \frac1{R_e} &\text{where the sum runs over all edges $e$}\\[-1.2mm]& \text{between the vertices $i$ and $j$\footnote{By convention, an empty sum is $0$.},}\\
a_{ii} &\text{if $i=j$}
\end{cases}
$$
where the diagonal values $a_{ii}$ are chosen such that the sum of all rows of $L$ vanishes.
The weight of a subgraph $H$ of $G$ is defined as
$$
\Pi(H):=\prod_{\text{$e$ an edge of $H$}}\frac1{R_e}\,.
$$
Then we recall the following:
\begin{proposition}\label{prop-matrix-tree}
\begin{enumerate}
\item The value of each cofactor $L_{ij}$ of $L$ is the sum of the weighs of all spanning trees of $G$.\label{qp}
\item If $L_{ij,ij}$ denotes the determinant of the matrix $L$ with rows $i,j$ and columns $i,j$ deleted,
then the quotient $L_{ij,ij}/R_e$ equals the sum of the weights of all spanning trees of $G$ which
contain the edge $e$ between the vertices $i$ and $j$.
\end{enumerate}
\end{proposition}
{\em Proof.}
The first part of the proposition follows directly from a general version of
the Matrix Tree Theorem (see, e.g., \cite[Theorem VI.27]{tutte}).

For the second part we proceed as follows: Let $S(G)$ denote the set of all spanning trees of $G$.
Observe, that by using the edge $e$ between the vertices $i$ and $j$ we can split up the sum of the weights of all spanning trees of $G$  as follows:
\begin{equation}\label{eq-spanning}
\sum_{T \in S(G)}\Pi(T) = \sum_{T \in S(G) \atop e \in T} \Pi(T) + \sum_{T \in S(G) \atop e \notin T} \Pi(T).
\end{equation}
Furthermore, the second sum on the right-hand side of~(\ref{eq-spanning})  corresponds to the sum of the weights of all spanning trees of $G-e$, i.e. $G$
with edge $e$ removed. Using the first part of the proposition, we get from~(\ref{eq-spanning}) that
$$
\sum_{T \in S(G) \atop e \in T} \Pi(T) = L_{ii} - L^e_{ii},
$$
where $L^e=(\ell^e_{hk})$ is the weighted Laplace matrix of $G-e$. 
We have
$$
\ell^e_{ij}=\ell^e_{ji}=\ell_{ij}+\frac1{R_e},\quad \ell^e_{ii}=\ell_{ii}-\frac1{R_e},\quad \ell^e_{jj}=\ell_{jj}-\frac1{R_e}
$$
and $\ell^e_{hk}=\ell_{hk}$ for all other $h,k$.
The term $L_{ii} - L^e_{ii}$ can be computed using Laplace's cofactor expansion. Expanding both $L_{ii}$ and $L^e_{ii}$ along the $j$-th row yields
$$
\sum_{T \in S(G) \atop e \in T} \Pi(T) = \sum_{k\neq i}\ell_{jk}(L_{ii})_{jk} - \sum_{k\neq i}\ell^e_{jk}(L^e_{ii})_{jk}.
$$
Since the cofactors $(L_{ii})_{jk}$ and $(L^e_{ii})_{jk}$ are equal for  all $k$ we are left with
\begin{equation}
\sum_{T \in S(G) \atop e \in T} \Pi(T) =(\ell_{jj}-\ell^e_{jj})  (L_{ii})_{jj}= \frac{L_{ij,ij}}{R_{e}}.\tag*{$\Box$}
\end{equation}
{\bf Remark.} The second part of Proposition~\ref{prop-matrix-tree} follows also quite 
easily from the All Minors Matrix Tree Theorem (see~\cite{chaiken}).

The connection to the equivalent resistance is given by the following
\begin{proposition}\label{prop-resistance}
The equivalent resistance $r_e$ over the edge $e$ connecting the vertices $i$ and $j$
is given by
\begin{equation}\label{eq-res}
r_e=\frac{L_{ij,ij}}{L_{11}}\,.
\end{equation}
\end{proposition}
{\em Remark.} Recall, that by Proposition~\ref{prop-matrix-tree}(\ref{qp}) the denominator in~(\ref{eq-res}) can be replaced by
any other cofactor $L_{hk}$.
\begin{proof}
Observe, that the Laplace matrix of the weighted multigraph $G$ corresponds
to the Laplace matrix of a weighted simple graph $H$ where the multiple
edges $e_1,\ldots e_k$ between each two vertices $i$ and $j$ of $G$ are collapsed to a single edge $e$ with weight 
$$R_e=\frac1{\frac1{R_{e_1}}+ \ldots+\frac1{R_{e_k}}}.$$
However this value corresponds exactly to the equivalent resistance 
of the parallel resistors $R_{e_1},\ldots,R_{e_k}$. Thus,
the equivalent resistance over the vertices $i$ and $j$ in $G$
equals the equivalent resistance over the vertices $i$ and $j$ in $H$
and the claim follows from~\cite[Lemma 2]{gupta} and Proposition~\ref{prop-matrix-tree}. 
\end{proof}

From now on we assume that is $G$ a finite planar multigraph with dual graph $G'$.
Recall that in general $G'$  depends on the embedding of $G$ in the plane.
\begin{definition}\label{def-dual}
Let $G'$ be the dual of a planar embedded multigraph $G$, and
let $G$ be interpreted as an electrical network by associating 
to each edge $e$ a resistance $R_e>0$.
For each edge $e$ of $G$, we define 
the  electrical resistance $R_{e'}$ of the dual edge $e'$
to be the conductance of $e$, i.e. $R_{e'}:=1/R_e$. 
Then, $G'$ equipped with these resistances is called the
{\em dual electrical network\/} of $G$.
\end{definition}
Observe, that the Laplace matrix $L'=(\ell'_{ij})$ of the dual electrical network $G'$ is given by
$$
\ell'_{ij}:=\begin{cases}
-\sum \frac1{R_e'}=-\sum  R_{e} &\text{where the sum runs over all edges $e'$}\\[-1.2mm]& \text{between the vertices $i$ and $j$ of $G'$,}\\
a_{ii} &\text{if $i=j$}
\end{cases}
$$
where the diagonal values $a_{ii}$ are chosen such that the sum of all rows of $L'$ vanishes.

Similarly the weight of a subgraph $H'$ of $G'$ is
$$
\Pi(H'):=\prod_{\text{$e'$ an edge of $H'$}}\frac1{R_{e'}} =\prod_{\text{$e'$ an edge of $H'$}}{R_{e}} \,.
$$
Then we have:
\begin{proposition}\label{prop-dual}
\begin{enumerate}
\item The value of an arbitrary cofactor $L'_{ij}$ of $L'$ is equal to the
sum of the weights of all spanning trees in $G'$ and there holds:
\begin{equation}\label{eq-weights}
L'_{ij}=L_{ij}\Pi(G')
\end{equation}
where $\Pi(G')=\prod R_k$ is the total weight of $G'$.
\item The product $R_e L'_{ij,ij}$ equals the sum of the weights of the spanning trees in $G'$
which contain the dual edge $e'$ of edge $e$.\label{prop-dual-ii}
\end{enumerate}
\end{proposition}
\begin{proof}
We only have to show equation~(\ref{eq-weights}), since all other statements
follow from Proposition~\ref{prop-matrix-tree}. Let $S(G)$ and $S(G')$ denote the set of all spanning trees of $G$ and its dual $G'$ respectively. Furthermore, let $\Psi$ denote the canonocal bijection from $S(G')$ to $S(G)$ given by $\Psi(T') = \{e \in G | e' \in G' - T' \}$. Observe, that the weight of a spanning tree $T'$ of $G'$ can be expressed as
\begin{equation}\label{eq-pit}
\Pi(T') = \frac{\Pi(G')}{\Pi(G'-T')}.
\end{equation}
Using the bijection $\Psi$ and Definition~\ref{def-dual} in~(\ref{eq-pit}), namely the fact, that the 
electrical resistance of an edge in $G'$ is equal to the conductance of the dual 
edge in $G$, we get
\begin{equation}\label{eq-pitt}
\Pi(T') = \Pi(G')\Pi(\Psi(T')).
\end{equation}
The characterization of $L_{ij}$ and $L'_{ij}$ as the sum of the weights of all 
spanning trees of $G$ and $G'$ respectively leads, together with~(\ref{eq-pitt}), to
\begin{align*}
\qquad L'_{ij} = \sum_{T' \in S(G')} \Pi(T') &= \sum_{T' \in S(G')} \Pi(G')\Pi(\Psi(T')) =\\&= \Pi(G')\sum_{T \in S(G)} \Pi(T) = \Pi(G')L_{ij}, \qquad
\end{align*}
where we have used the bijectivity of $\Psi$ in the penultimate equality. This completes the proof.
\end{proof}
\section{The sum formula in dual networks}
The main result is now the following:
\begin{theorem}\label{thm-main}
Let $R_e$ be the resistance of an edge $e$ and $R_{e'}=1/R_e$ the resistance
of the dual edge $e'$ in the dual electrical network. Let $r_e$ denote the equivalent resistance over edge $e$
and $r_{e'}$ denote the equivalent resistance over edge $e'$. Then
$$
\frac{r_e}{R_e}+\frac{r_{e'}}{R_{e'}}=1.
$$
\end{theorem}
For a proof of this formula based upon physical arguments see~\cite{furrer}. 
Here, we provide a purely graph theoretic proof.
\begin{proof}
Let $e$ be an edge between the vertices $i$ and $j$ in $G$, and we may
assume that the vertices are numbered such that $e'$ also runs between
the vertices $i$ and $j$ in $G'$.
In a first step, we are going to derive a new expression for $L'_{ij,ij}$.
Let $S(G)$ and $S(G')$ denote the set of all spanning trees of $G$ and its dual $G'$ respectively, and let $\Psi$ be the canonical bijection from $S(G')$ to $S(G)$ as defined in the proof of Proposition~\ref{prop-dual}. By part~\ref{prop-dual-ii} of Proposition~\ref{prop-dual} we have
\begin{equation}\label{eq-o}
L'_{ij,ij} = \frac{1}{R_{e}}\sum_{T'\in S(G') \atop e'\in T'}\Pi(T').
\end{equation}
The identity~(\ref{eq-pitt}) and the fact, that $\Psi(T')$ does not contain the edge $e$ if $T'$ contains the dual edge $e'$ allow
to rewrite the right-hand side of~(\ref{eq-o}) as follows
\begin{equation}\label{eq-oo}
\frac{1}{R_{e}}\sum_{T'\in S(G') \atop e'\in T'}\Pi(T') = 
\frac{\Pi(G')}{R_{e}}\sum_{T \in S(G) \atop e \notin T}\Pi(T).
\end{equation}
Furthermore, the sum on the right-hand side of~(\ref{eq-oo}) can be expressed as the difference between the sum of the weights of all spanning trees of $G$ and the sum of the weights of the spanning trees, that contain the edge $e$. Therefore, it holds that
\begin{equation}\label{eq-ooo}
L'_{ij,ij} = \frac{\Pi(G')}{R_{e}}\Bigl(\sum_{T \in S(G)}\Pi(T) - 
\sum_{T \in S(G) \atop e \in T}\Pi(T)\Bigr).
\end{equation}
Using Proposition~\ref{prop-matrix-tree} in~(\ref{eq-ooo}) yields the following identity:
\begin{equation}\label{eq-oooo}
L'_{ij,ij} = \frac{\Pi(G')}{R_{e}}\Bigl(L_{ii} - \frac{L_{ij,ij}}{R_{e}}\Bigr).
\end{equation}
Now, in a second step, it follows from Proposition~\ref{prop-resistance} and Definition~\ref{def-dual}, that
\begin{equation}\label{eq-ooooo}
\frac{r_e}{R_e}+\frac{r_{e'}}{R_{e'}} = \frac{L_{ij,ij}}{L_{ii}R_{e}} + 
R_{e}\frac{L'_{ij,ij}}{L'_{ii}}\,.
\end{equation}
Using the first part of Proposition~\ref{prop-dual} and~(\ref{eq-oooo}), we can rewrite the right-hand side of~(\ref{eq-ooooo})
and simplify the resulting expression to arrive at
$$
\frac{r_e}{R_e}+\frac{r_{e'}}{R_{e'}} = \frac{L_{ij,ij}}{L_{ii}R_{e}} +
R_{e}\frac{\frac{\Pi(G')}{R_{e}}(L_{ii} - \frac{L_{ij,ij}}{R_{e}})}{L_{ii}\Pi(G')}=1,
$$
as claimed.
\end{proof}
\begin{example}
Let us consider the following electrical network:
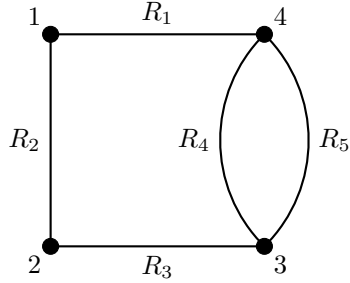
\begin{center}
\begin{tikzpicture}[x=80,y=80]
\draw[fill=black] (0,1) circle (3pt);
\draw[fill=black] (0,0) circle (3pt);
\draw[fill=black] (1,0) circle (3pt);
\draw[fill=black] (1,1) circle (3pt);
\draw (0,1) node[anchor=south east] {$1$};
\draw (0,0) node[anchor=north east] {$2$};
\draw (1,0) node[anchor=north west] {$3$};
\draw (1,1) node[anchor=south west] {$4$};
\draw [line width=.8pt] (1,1) -- node[above] {$R_1$} (0,1)-- node[left] {$R_2$}  (0,0) -- node[below] {$R_3$} (1,0);
\draw [line width=.8pt] (1,0) to[out=45,in=-45] node[right] {$R_5$}  (1,1);
\draw [line width=.8pt] (1,0) to[out=135,in=-135] node[left] {$R_4$}(1,1);
\end{tikzpicture}
\end{center}
The corresponding Laplace matrix $L$ is
$$
L=\begin{pmatrix}
\frac1{R_1}+\frac1{R_2} & -\frac1{R_2} & 0 &-\frac1{R_1}\\
-\frac1{R_2} & \frac1{R_2}+\frac1{R_3} & -\frac1{R_3} & 0\\
0 & -\frac1{R_3} &\frac1{R_3}+\frac1{R_4}+\frac1{R_5}& -\frac1{R_4}-\frac1{R_5}\\
-\frac1{R_1}&0&-\frac1{R_4}-\frac1{R_5}& \frac1{R_1}+\frac1{R_4}+\frac1{R_5}
\end{pmatrix}
$$
The cofactor
$$
L_{11}=\frac1{R_1R_2R_3R_4R_5}\bigl(R_4(R_1+R_2+R_3) +R_5(R_1+R_2+R_3+R_4)\bigr)
$$
corresponds indeed to the total weight of all spanning trees of $G$ as one easily checks directly.
For the edge $e$ between the vertices $3$ and $4$ with resistance $R_4$,
we get
$$
L_{34,34}=\frac1{R_1R_2}+\frac1{R_1R_3}+\frac1{R_2R_3}
$$
which, divided by $R_4$, gives the sum of the weights of the trees which contain $e$,
as stated in Proposition~\ref{prop-matrix-tree}.

The dual network looks as follows:
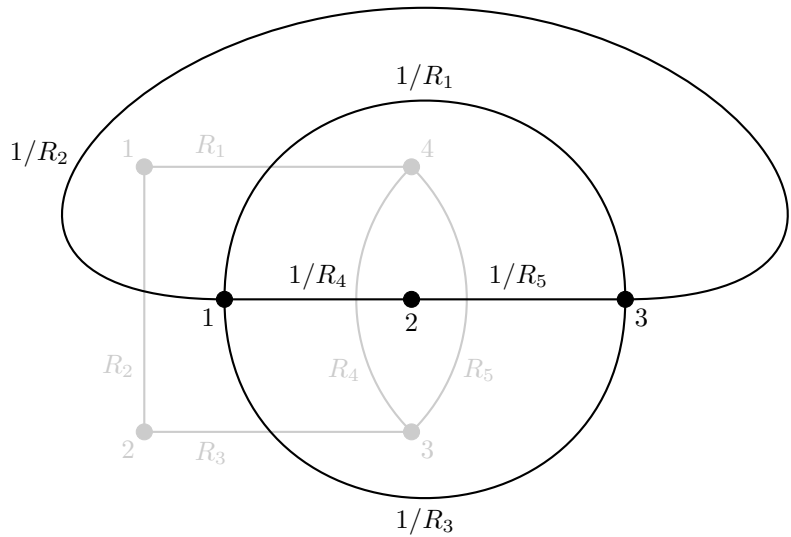
\begin{center}
\definecolor{gray}{rgb}{.8,.8,.8}
\begin{tikzpicture}[x=100,y=100]
\clip(-.5,-.4) rectangle (2.45,1.65);
\draw[color=gray, fill=gray] (0,1) circle (3pt);
\draw[color=gray, fill=gray] (0,0) circle (3pt);
\draw[color=gray, fill=gray] (1,0) circle (3pt);
\draw[color=gray, fill=gray] (1,1) circle (3pt);
\draw[color=gray] (0,1) node[anchor=south east] {$1$};
\draw[color=gray] (0,0) node[anchor=north east] {$2$};
\draw[color=gray] (1,0) node[anchor=north west] {$3$};
\draw[color=gray] (1,1) node[anchor=south west] {$4$};
\draw [color=gray, line width=.8pt] (1,1) -- node[above, style={pos=.75}] {$R_1$} (0,1)-- node[left, style={pos=.75}] {$R_2$}  (0,0) -- node[below, style={pos=.25}] {$R_3$} (1,0);
\draw [color=gray, line width=.8pt] (1,0) to[out=45,in=-45] node[right, style={pos=.25}] {$R_5$}  (1,1);
\draw [color=gray, line width=.8pt] (1,0) to[out=135,in=-135] node[left, style={pos=.25}] {$R_4$}(1,1);
\draw[color=black, fill=black] (1,.5) circle (3pt);
\draw[color=black, fill=black] (.3,.5) circle (3pt);
\draw[color=black, fill=black] (1.8,.5) circle (3pt);
\draw [color=black, line width=.8pt] (.3,.5) -- node[above] {$1/R_4$} (1,.5) -- node[above] {$1/R_5$} (1.8,.5);
\draw [color=black, line width=.8pt] (.3,.5) to[out=90,in=90,distance=100] node[above, style={pos=.5}] {$1/R_1$}(1.8,.5);
\draw [color=black, line width=.8pt] (.3,.5) to[out=-90,in=-90,distance=100] node[below, style={pos=.5}] {$1/R_3$}(1.8,.5);
\draw [color=black, line width=.8pt] (.3,.5) to[out=-180,in=180,distance=120] node[left, style={pos=.5},xshift=-2] {$1/R_2$}(1.05,1.6) to[out=0,in=0,distance=120] (1.8,.5);
\draw[color=black] (.3,.5) node[anchor=north east] {$1$};
\draw[color=black] (1,.5) node[anchor=north,yshift=-2] {$2$};
\draw[color=black] (1.8,.5) node[anchor=north west] {$3$};

\end{tikzpicture}

\end{center}
and the corresponding Laplace matrix is
$$
L'=\begin{pmatrix}
R_1+R_2+R_3+R_4 & -R_4 & -(R_1+R_2+R_3)\\
-R_4&R_4+R_5 & -R_5\\
-(R_1+R_2+R_3) & -R_5 &R_1+R_2+R_3+R_5
\end{pmatrix}.
$$
The cofactor
$$
L'_{11}=R_4(R_1+R_2+R_3) +R_5(R_1+R_2+R_3+R_4)
$$
is the total weight of the spanning trees of $G'$. And indeed, we have
$$
L'_{11}=L_{11}R_1R_2R_3R_4R_5
$$
as predicted by Proposition~\ref{prop-dual}. Furthermore, we get
$$
L'_{12,12}=R_1+R_2+R_3+R_5
$$
which gives according to Proposition~\ref{prop-matrix-tree}, after multiplication 
by $R_4$, the total weight of the trees in $G'$ which contain the dual edge $e'$.

Now, the equivalent resistances over edge $e$ and $e'$ respectively are, according to Proposition~\ref{prop-resistance},
$$\arraycolsep=2pt\def\arraystretch{2.2}
\begin{array}{lll}
r_4&=\displaystyle{\frac{L_{34,34}}{L_{11}} }&=\displaystyle{ \frac{R_4R_5(R_1+R_2+R_3)}{R_4(R_1+R_2+R_3) +R_5(R_1+R_2+R_3+R_4)}}\\
r'_4&=\displaystyle{\frac{L'_{12,12}}{L'_{11}}}&= \displaystyle{\frac{R_1+R_2+R_3+R_5}{R_4(R_1+R_2+R_3) +R_5(R_1+R_2+R_3+R_4)}}
\end{array}
$$
and finally indeed, with $R'_4=1/R_4$,
$$
\frac{r_4}{R_4}+\frac{r'_4}{R'_4}=1.
$$
\end{example}

\bibliographystyle{plain}

\begin{thebibliography}{1}

\bibitem{bapat}
Ravindra~B. Bapat.
\newblock {\em Graphs and matrices}.
\newblock Universitext. Springer, London; Hindustan Book Agency, New Delhi,
  second edition, 2014.

\bibitem{gupta}
Ravindra~B. Bapat and Somit Gupta.
\newblock Resistance distance in wheels and fans.
\newblock {\em Indian J. Pure Appl. Math.}, 41(1):1--13, 2010.

\bibitem{chaiken}
Seth Chaiken.
\newblock A combinatorial proof of the all minors matrix tree theorem.
\newblock {\em SIAM J. Algebraic Discrete Methods}, 3(3):319--329, 1982.

\bibitem{furrer}
Martina Furrer.
\newblock {Widerstandssumme in dualen Netzwerken}.
\newblock {Mentoriere Arbeit}, 2017.
\newblock {ETH Z\"urich}.

\bibitem{tutte}
William~T. Tutte.
\newblock {\em Graph theory}, volume~21 of {\em Encyclopedia of Mathematics and
  its Applications}.
\newblock Cambridge University Press, Cambridge, 2001.
\newblock With a foreword by Crispin St. J. A. Nash-Williams, Reprint of the
  1984 original.

\bibitem{yang2013}
Yu~Jun Yang.
\newblock An identity on resistance distances.
\newblock In {\em Material and Manufacturing Technology IV}, volume 748 of {\em
  Advanced Materials Research}, pages 1024--1027. Trans Tech Publications, 10,
  2013.

\end{thebibliography}

\end{document}